\documentclass[11pt]{article}

\usepackage{amssymb}
\usepackage{amsmath}
\usepackage{amsthm}
\usepackage[colorlinks=true,citecolor=black,linkcolor=black,urlcolor=blue]{hyperref}
\newcommand{\mc}[1]{\mathcal{#1}}

\begin{document}

\parskip 1ex            
\parindent 5ex		

\newcommand{\rbrac}[1]{\left(#1\right)}
\newcommand{\sbrac}[1]{\left[ #1\right]}
\newcommand{\cbrac}[1]{\left\{ #1\right\}}
\newtheorem{theorem}{Theorem}[section]
\newtheorem{lemma}[theorem]{Lemma}
\newtheorem{conjecture}[theorem]{Conjecture}
\newtheorem{proposition}[theorem]{Proposition}
\newtheorem{corollary}[theorem]{Corollary}
\newtheorem{claim}[theorem]{Claim}
\newcommand{\G}{\Gamma}

\title{The sum-free process}
\author{
Patrick Bennett\thanks{Mathematics Department, Western Michigan University, Kalamazoo, Mi, 49008. Email: {\tt patrick.bennett@wmich.edu}. Research supported in part by Simons Foundation Grant \#426894.
}}
\date{}
\maketitle

\begin{abstract}
$S \subseteq \mathbb{Z}_{2n}$ is said to be sum-free if $S$ has no solution to the equation $a+b=c$. The sum-free process on $\mathbb{Z}_{2n}$ starts with $S:=\emptyset$, and iteratively inserts elements of $\mathbb{Z}_{2n}$, where each inserted element is chosen uniformly at random from the set of all elements that could be inserted while maintaining that $S$ is sum-free. We prove a lower bound (which holds with high probability) on the final size of $S$, which matches a more general result of Bennett and Bohman (\cite{Hind}), and also matches the order of a sharp threshold result proved by Balogh, Morris and Samotij (\cite{BMSsumfree}). We also show that the set $S$ produced by the process has a particular non-pseudorandom property, which is in contrast with several known results about the random greedy independent set process on hypergraphs. 
\end{abstract}

\section{Introduction}

Let $\mathcal{H}$ be a hypergraph. A set of vertices $S \subset V(H)$ is called \emph{independent} if $S$ contains no edge of $\mathcal{H}$. The random greedy algorithm for independent sets starts with $S= \emptyset$, and then randomly inserts elements into $S$ so long as $S$ remains independent. Specifically, at the $i^{th}$ step, we put $S:=S \cup \{v_i \}$ where $v_i$ is chosen uniformly at random from all vertices $v$ such that $S \cup \{v \}$ is independent (we halt when there are no such vertices $v$). The algorithm terminates with a maximal independent set. 

One notable instance of this algorithm is the $H$-free process for any fixed $k$-uniform hypergraph $H$ (where $k \ge 2$ so $H$ might be a graph). The process starts with an empty hypergraph on vertex set $[n]$, and iteratively inserts randomly chosen hyperedges  so long as we never create a subhypergraph isomorphic to $H$. The $H$-free process can be viewed as an instance of the random greedy independent set algorithm running on a hypergraph $\mathcal{H}$ with vertex set $\binom{[n]}{k}$, where each edge of the hypergraph $\mathcal{H}$ corresponds to a subset of $\binom{[n]}{k}$ forming a copy of $H$. 

 Bohman \cite{Bo10} analyzed the $H$-free process when $H$ is the graph $K_3$, determining (up to a constant) how long the process lasts , and bounding the independence number of the final graph formed. Bohman's analysis of the process included information about the independence number of the graph produced, and gave a second proof that the Ramsey number $R(3,t)$ is $\Omega \left( \frac{t^2}{\log t} \right)$. In the same paper, Bohman considered the $K_4$-free process and improved the best known lower bound on $R(4,t)$. Bohman and Keevash \cite{Hfree} went on to analyze the $H$-free process for many other graphs $H$, resulting in new lower bounds on $R(s,t)$ for fixed $s \ge 5$, and new lower bounds on the Tur\'{a}n numbers of certain bipartite graphs. More recently, Bohman and Keevash \cite{BKtridynamic}, and independently Fiz Pontiveros,  Griffiths, and Morris \cite{FGM} proved that the $K_3$-free process terminates with $\left(\frac{1}{2 \sqrt{2}} + o(1) \right) \log^{1/2} n \cdot n^{3/2}$ edges (and also gave better bounds on the indpendence number of the graph produced by the $K_3$-free process). Other specific instances of $H$-free processes have been further studied, on some of which there are upper and lower bounds (matching up to a constant factor) on the final size of the graph, notably the $K_4$-free process by Warnke \cite{Wa3} and independently Wolfovitz \cite{Wz2}; the $C_\ell$-free process by Warnke \cite{Wa2} and independently  Picollelli \cite{P3}; and the $(K_4-e)$-free process (also known as the diamond-free process) by Picollelli \cite{P1}. With Bohman, the author in \cite{Hind} considers the independent set process on a general class of hypergraphs, and proved a lower bound (on the size of the final independent set) which generalizes many of the known lower bounds for specific hypergraphs. The general bound in \cite{Hind} applies to some instances of the $H$-free process where $H$ is a hypergraph, as well as the $k$-AP-free process (which chooses elements of $\mathbb{Z}_n$ while avoiding a $k$-term arithmetic progression).

\begin{theorem}
\label{theory:main}
Let $r$ and $ \epsilon > 0 $ be fixed.  Let 
$ {\mathcal H} $ be a $ r$-uniform, $D$-regular hypergraph on $N$ 
vertices such that $ D > N^{\epsilon} $.  Define the {\em degree} of a set $ A \subset V(\mathcal{H})$ to be the number of edges of 
$ {\mathcal H} $ that contain $A$.
For $ a = 2, \dots, r-1 $ we define $ \Delta_a( {\mathcal H}) $
to be the maximum degree of $A$ over $A \in \binom{V}{a} $.  We also
define the {\em $b$-codegree} of a pair of distinct vertices $ v,v'$ to be 
the number of edges $e,e' \in {\mathcal H} $ such that $ v \in e \setminus e', v' \in e' \setminus e$ 
and $ |e \cap e'|=b $.  We let $ \G_b ( \mc{H}) $ be the maximum 
$b$-codegree of $ {\mathcal H} $. 

If
\begin{equation}
\label{eq:degcond}
 \Delta_\ell( {\mathcal H} ) < D^{ \frac{r-\ell}{r-1} - \epsilon} \ \ \ \text{ for }
\ell = 2, \dots, r-1 
\end{equation}
and 
\begin{equation}
 \Gamma_{r-1}( {\mathcal H} ) < D^{1- \epsilon} \label{eq:codegcond}
 \end{equation}
then the random greedy independent set
algorithm produces an independent set $S$ in $ {\mathcal H} $ with
\begin{equation}
\label{eq:lowerbound}
|S| = \Omega\left( N \cdot \left( \frac{\log N}{ D} \right)^{\frac{1}{r-1}} \right)
\end{equation}
with probability $ 1 - \exp\left\{ - N^{\Omega(1)} \right\} $.
\end{theorem}

The indpendent sets produced by the algorithm tend to have pseudorandom properties. For example, the $K_3$-free process produces a graph whose independence number is roughly the same as it would be in a random graph with the same edge density (\cite{Bo10}, \cite{BKtridynamic}, \cite{FGM}), and for any fixed $K_3$-free graph $G$, the number of copies of $G$ in the graph produced by the $K_3$-free process is roughly the same as it would be in a random graph with the same edge density (see \cite{GW}).  \cite{Hind} also has a pseudorandom type result for the independent sets produced by the algorithm, which generalized Wolfovitz's result and which they used to bound the Gowers norm of the set produced by the $k$-AP-free process. We state this result from \cite{Hind} now.

\begin{theorem}
\label{theory:count}
Fix $s$ and a $s$-uniform hypergraph $ {\mathcal G} $ on vertex set $V(\mathcal{H})$ (i..e the 
same vertex set as the hypergraph $\mc{H}$).  We let $ X_\mc{G} $ be the number of edges in $\mc{G}$ 
that are contained in the independent set produced at the $i^{th}$ step of the random greedy process on $ \mc{H} $.  Set
$ p = p(i) = i/N$ and let $ i_{\rm max} $ be the lower bound (\ref{eq:lowerbound}) on the size 
of the independent set given by the 
random greedy algorithm given in Theorem~\ref{theory:main}.
If no edge of $ \mc{G}$ contains an edge of $ \mc{H}$, $ i < i_{\rm max}$ is fixed, $ | \mc{G}| p^s \to \infty $ and 
$ \Delta_a( \mc{G}) = o( p^a |\mc{G}|) $ for $ a =1, \dots, s-1$ then
\[ X_\mc{G} = | \mc{G}| p^{s} (1+o(1)). \]
with high probability.
\end{theorem}

This paper addresses  the sum-free process. In this process we look for a set $S \subset \mathbb{Z}_{2n}$ such that $S$ has no solutions to the equation $a+b=c$. Define our edge set $E$ to be the family of all solutions $\{a,b,c\}$ to $a+b=c$ (Such edges $\{a, b, c\}$ are called \emph{Schur triples}). Note that $\{a, b, c\}$ may have $1, 2,$ or $3$ distinct elements. We write the generic form $\{a, b, c\}$ with the understanding that if, say, $b=c$ then we mean the set $\{a,b\}$ and not a multiset with two copies of $b$. Thus, $S$ is sum-free if and only if $S$ is an independent set in the hypergraph $\mathcal{H}$ with vertex set $\mathbb{Z}_{2n}$ and edge set $E$. $\mc{H}$ is not uniform but almost all of the edges have size $3$. Indeed, each vertex $v \neq 0$ is in $O(1)$ edges of size $2$ and no edge of size $1$.

Theorem \ref{theory:main} cannot be applied directly to the sum-free process, since $\mathcal{H}$ is not quite uniform, and also not quite regular. However it is nearly uniform and nearly regular, and these issues alone would not present much difficulty in the analysis. But there is a more crucial way in which $\mathcal{H}$ fails to satisfy the hypotheses of Theorem \ref{theory:main}: the codegree condition \eqref{eq:codegcond}. Consider the vertices $v$ and $-v$ for $v \neq 0, n$. Each of $v, -v$ is in about $D = \Theta(n)$ edges, but the $2$-codegree of $v, -v$ is also $\Theta(n)$ since whenever we have the equation $v+b=c$ we also have the equation $-v + c = b$. Thus, one of the points of interest in this paper is to see how to overcome this issue.

The largest sum-free subset of $\mathbb{Z}_{2n}$ is the set consisting of the odd numbers. Indeed, Balogh, Morris, and Samotij in \cite{BMSsumfree} proved that $p=\left( \frac{\log n}{3n} \right)^{1/2}$ is a sharp threshold above which we have the following: if $G_p$ is a random subset of $\mathbb{Z}_{2n}$ where each element is included with probability $p$ independently, then w.h.p. the largest sum-free subset of $G_p$ is simply the set of its odd elements. Roughly speaking, the fact that \emph{odd + odd = even} imposes a very specific structure on maximum sum-free subsets of random sets in $\mathbb{Z}_{2n}$. However, this special structure ceases to be relevant for very sparse random sets in $\mathbb{Z}_{2n}$. 

Our main theorem is as follows: 

\begin{theorem}\label{thm:main}
With high probability, the sum-free process run on $\mathbb{Z}_{2n}$ produces a set of size at least $$i_0:= \displaystyle \frac{1}{6}  n^{1/2} \log^{1/2} n.$$
\end{theorem}

 We conjecture that a matching (up to a constant factor) upper bound holds, since intuitively the sum-free process should not produce sets that have density larger than the threshold in \cite{BMSsumfree}.
 
 Another point of interest in this paper is to demonstrate a non-pseudorandom statistic of the set $S$ produced by the sum-free process (in contrast to all the pseudorandom results concerning other instances of the hypergraph independent set process). It turns out that the set $S$ tends to contain a large number of pairs $v, -v$, much larger than one would expect in a (uniformly chosen) random set of size $|S|$. Indeed, if $|S| = \Theta(n^{1/2} \log^{1/2} n)$ were chosen uniformly at random, then we would expect $S$ to contain $\Theta(\log n)$ pairs $v, -v$. However, as we will see in the last section of the paper, the number of such pairs in $S$ actually grows like a power of $n$.

Our analysis of the sum-free process on $\mathbb{Z}_{2n}$ easily extends to  $\mathbb{Z}_{2n+1}$ with no substantial changes. Note that  $\mathbb{Z}_{2n+1}$ also has linear-sized sum-free subsets (e.g. the set of all odd numbers less than $n$). However, since there is no result analogous to \cite{BMSsumfree} for $\mathbb{Z}_{2n+1}$, we work in $\mathbb{Z}_{2n}$.

\section{The Algorithm and Associated Random Variables}

The random greedy algorithm for sum-free sets starts with a sum-free set $S(0):= \emptyset$, and a set $Q(0)$ of elements that may be inserted into the $S(0)$ without spoiling the sum-free property. Since $0+0=0$, the element $0$ cannot be part of any sum-free set. Any singleton except for $\{0\}$ is sum-free, so $Q(0) = \mathbb{Z}_{2n} \setminus \{0\}$. At step $i$, the algorithm selects an element $s(i)$ uniformly at random from $Q(i)$ and puts $S(i+1):=S(i) \cup \{s(i)\}$. The algorithm then determines  $Q(i+1)$, the set of elements that could potentially be inserted in $S(i+1)$ without spoiling the sum-free property. If $Q(i+1)$ is empty, the algorithm terminates with a maximal sum-free set. 

We call the elements of $S(i)$ \emph{chosen}, the elements of $Q(i)$ \emph{open}, and all other elements of $\mathbb{Z}_{2n}$ \emph{closed}. As we noted above, at the start of the algorithm $0$ is closed and every other element is open.



 We will define several more random variables which (roughly speaking) represent the degrees of vertices. 
 We will partition the set of edges containing $v$ according to what role $v$ plays in the corresponding equation $a+b=c$. For $v \in \mathbb{Z}_{2n}\setminus S(i) \setminus \{0\}$ and $k=1,2,3$ we define the random variables $D_{k,L} (v,i)$, the set of edges $e \in E$ such that $v \in e$, $e \setminus \{v\} \subseteq Q(i) \cup S(i)$,  $|e \cap Q(i) \setminus \{v\}| = k-1$, and $v$ appears on the left of the equation corresponding to $e$ (i.e. $v$ plays the role of $a$ or $b$ in the equation $a+b=c$); and for $v \in \mathbb{Z}_{2n}\setminus S(i)$, we define $D_{k,R} (v,i)$, the set of edges $e \in E$ such that $v \in e$, $e \setminus \{v\} \subseteq Q(i) \cup S(i)$,  $|e \cap Q(i) \setminus \{v\}| = k-1$ and $v$ is on the right side of the equation. We also define the random variable $$D_2 (v,i):= \{q \in Q(i): q \neq v, \exists e \in  D_{2, L}(v, i) \cup D_{2, R}(v, i). q \in e\}.$$  
Several of the variables we just defined take two arguments (a vertex $v$ and the step $i$). Sometimes for shorthand we will suppress the  $i$.

\section{Heuristically anticipating the trajectories}

We use some heuristics to anticipate the likely values of the random variables throughout the process. Intuitively, we will mostly assume that certain aspects of the hypergraph $\mathcal{H}(i)$ are the same as what they would be if $S(i)$ were a set of $i$ uniformly chosen random elements (rather than having been chosen to satisfy the sum-free property), each chosen with probability $i/2n$. In other words we are using a heuristic assumption of pseudorandomness to guess the trajectories of our variables as the process evolves. But due to the special relationship between $v$ and $-v$ in this hypergraph, we will have to be careful about assuming independence: we will see that $u$ and $v$ behave almost independently except when $u=-v$. Furthermore, the relationship between $v$ and $-v$ will cause the vertex $0$ to behave completely differently from other vertices. 

Consider a vertex $v \neq 0$. Since $v$ is in about $3n$ edges and these edges do not interact pathologically (e.g. almost all pairs of these edges are disjoint other than $v$), the probability that $v$ is open (i.e. that there is no edge $\{v, s_1, s_2\}$ with $s_1, s_2 \in S$) is heuristically 
\[
\rbrac{ 1-\rbrac{\frac{i}{2n}}^2}^{3n} \approx e^{-\frac{3i^2}{4n} } = e^{-\frac{3}{4} t^2} =: p(t)
\]
where we define $$t=t(i):= n^{-1/2}i.$$

Now consider another vertex $u \neq v, -v, 0$. We would like to guess the probability that both $u$ and $v$ are open. We should be careful to take possible correlations into account, so we imagine that we know $u$ is open and consider the effect on the probability $v$ is open. 
It turns out that the relevant thing to look for is pairs of edges $e \ni u$, $e' \ni v$ with $|e\cap e'|=2$. The pairs with $|e\cap e'|=2$ are worrisome since if one of $u, v$ gets closed by one of the edges then the other vertex also gets closed automatically. On the other hand since we know $u$ is open we already know that $v$ does not get closed via any such edge. But since there are few such edges, knowing that $u$ is open does not affect many of $v$'s chances to get closed. We estimate the probability that both $u$ and $v$ are open to be 
\[
P(u \mbox{ open}) \cdot P(v \mbox{ open}| u \mbox{ open})\approx \rbrac{ 1-\rbrac{\frac{i}{2n}}^2}^{3n}\rbrac{ 1-\rbrac{\frac{i}{2n}}^2}^{3n-O(1)} \approx p^2,
\]
so heuristically $u$ and $v$ behave independently. Thus we are ready to make the following predictions for $v \neq 0$ (note that the probability that an element is in $S$ is $i/2n= t/2n^{1/2}$): \newpage
$$Q \approx 2np \;\;\;\;\;\;\;\;\;\;\;\;\;\;\;\;\;\;\;\;\;\;\;\;  D_{3,L}(v) \approx 2np^2 \;\;\;\;\;\;\;\;\;\;\;\;\;\;\; D_{3,R}(v) \approx np^2 \;\;\;\;\;\;\;\;\;\;\;\;\;\;\;$$ $$D_{2,L}(v) \approx 2n^{1/2} tp\;\;\;\;\;\;\;\;\;\;\;\;\;\;\;D_{2,R}(v) \approx n^{1/2} tp.$$ 
Note that we are using our heuristic assumption of independence, for example when we say $ D_{3,R}(v) \approx np^2$ and it is important to note that since $v \neq 0$, $v$ is not in any equation $a+b=v$ such that $a=-b$ (and so $a$ and $b$ being open are treated as independent events).

Now we would like to guess the probability that both $v$ and $-v$ are open. So we imagine we know $-v$ is open, and we consider pairs of edges $e \ni v$, $e' \ni -v$ with $|e\cap e'|=2$. Now there are about $2n$ such pairs and we know that $v$ cannot be closed via any such edge, meaning there are only about $3n-2n=n$ edges that might still close $v$.  Thus we estimate the probability that both $v$ and $-v$ are open to be 
\[
P(-v \mbox{ open}) \cdot P(v \mbox{ open}| -v \mbox{ open})\approx \rbrac{ 1-\rbrac{\frac{i}{2n}}^2}^{3n}\rbrac{ 1-\rbrac{\frac{i}{2n}}^2}^{n} \approx p^{4/3}.
\]
Here we can begin to see the mildly pathological behavior of this process emerge: if $-v$ is open then $v$ is more likely to be open. We make the prediction $$D_{3, R}(0) \approx np^{4/3}.$$ Now we predict $D_{2, R}(0) $ by guessing the probability that $v$ is open given that $-v$ is chosen. If $-v$ is chosen then there are about $2n$ edges $e' \ni -v$ that have another edge $e \ni v$ with $|e \cap e'|=2$ and $v$ cannot be closed by any such edge. Thus we estimate the probability that $v$ is open and $-v$ is chosen to be 
\[
P(-v \mbox{ chosen}) \cdot P(v \mbox{ open}| -v \mbox{ chosen})\approx \frac{i}{2n} \rbrac{ 1-\rbrac{\frac{i}{2n}}^2}^{n} \approx \frac12 n^{-1/2}tp^{1/3}.
\]
Thus we make the prediction $$D_{2, R}(0) \approx n^{1/2}tp^{1/3},$$ which we note is larger than our prediction for $D_{2, R}(v)$ for any $v \neq 0$. Roughly speaking this means that if $-v$ is chosen then it is more likely $v$ will stay open for longer than it otherwise would, and intuitively this should mean that $v$ is more likely to eventually be chosen, so we should end up with a lot of pairs $\{v, -v\}$ in $S$. We will now see that this is the case. 
We predict the value of $D_{1, R}(0)$. At each step, the variable $D_{1, R}(0)$ either stays the same or increases by $1$. The probability of increasing is $$\frac{D_{2,R}(0)}{Q} \approx \frac{1}{2} n^{-1/2} t p^{-2/3}$$
integrating the above expression gives the prediction:

$$D_{1, R}(0) \approx \frac{1}{2} \left( p^{-2/3}-1\right)  \approx \frac{1}{2}  p^{-2/3}.$$
The variables $D_{1, R}(v)$ for $v \neq 0$ will be much smaller, and we will prove that later. 

We finish this section with a general conjecture about the random greedy algorithm for independent sets in hypergraphs. This conjecture ought to apply to all sufficiently ``nice" (almost uniform, almost regular, not too sparse) hypergraphs. In general, the algorithm should terminate soon after the number of open elements is negligible compared to the number we've already chosen. Our conjecture is that we can predict when this happens using the heuristically derived trajectory for $Q$. 
\begin{conjecture}\label{conj}
The step $i$ at which the random greedy independent set algorithm terminates w.h.p. is asymptotically the value $i$ such that $Q \approx i$. 
In particular, the sum-free process terminates when $2np \approx n^{1/2} t$, so when $$i = \left(\sqrt{\frac{2}{3}}+o(1) \right) n^{1/2} \log^{1/2} n$$
\end{conjecture}
Even proving the lower bound of the conjecture presents considerable difficulty, particularly due to the fact that we suspect $D_{1, R}(0)$ becomes larger than $D_2(v)$ by that time (see equation \eqref{d2est} to see why this would wreak havoc on our present analysis of the process). 

\section{Proof Overview}
\label{sec:proof}

We appeal to the usual differential equations method to establish dynamic concentration of the random variables around the trajectories we heuristically derived. See Wormald's survey \cite{nick2} for a nice introduction to the standard method and a general theorem about random variables following trajectories given by differential equations. See also Warnke's note \cite{Wanote} for a nice short proof of a stronger version of Wormald's general theorem. In the standard differential equations method we define a ``good event" stipulating that our random variables are close to their trajectories, and a family of martingales such that if a random variable strays far from its trajectory then one of our martingales has a large deviation, which is unlikely. This method is employed by many of the previously mentioned papers including \cite{BaBe, Hind, Bo10, Hfree, BKtridynamic, FGM, P1, P3, Wa2, Wa3, Wz2, GW}. For two more more recent uses of the method see Bohman and Warnke \cite{Bowa} as well as Bennett, Dudek and Zerbib \cite{BDZ}. Our present proof most resembles \cite{Bo10}.

Define $$i_0:= \displaystyle \frac{1}{6}  n^{1/2} \log^{1/2} n. $$ 
Define the stopping time $T$ as the minimum of $i_0$ and the first step $i$ that any of the following bounds fail: 

\begin{eqnarray}
& \left|D_{3,L}(v) - 2np^2\right| \leq 2f_3   \;\;\;\; \forall v \notin \pm S(i) \cup \{0\}, &\left|D_{2, L}(v) - 2n^{1/2}tp\right| \leq 2f_2  \;\;\;\; \forall v \notin \pm S(i) \cup \{0\}, \nonumber\\
&  \left|D_{3,R}(v) - np^2\right| \leq f_3    \;\;\;\;  \forall v \notin S(i)\cup \{0\},  &  \left|D_{2,R}(v) -  n^{1/2}tp\right| \leq f_2  \;\;\;\;  \forall v \notin S(i)\cup \{0\},\nonumber\\
& \left|D_{3,R}(0) - np^{4/3}\right| \leq f_{3, 0}, &  \left|D_{2,R}(0) - n^{1/2}tp^{1/3} \right| \leq f_{2, 0},  \nonumber\\
&D_{1, L}(v), D_{1, R}(v) \leq \log^2 n \;\;\;\;   \forall v \neq 0, &  \left|D_{1, R}(0) - \frac{1}{2}  p^{-2/3} \right| \leq f_{1, 0},  \nonumber\\
& \left|Q - 2n p  \right| \leq  f_q. \label{estimates}
\end{eqnarray}

Theorem \ref{thm:main} is proved in sections 6-11.  We will show that, for specific choices of the error functions $f$, that the stopping time $T=i_0$  with high probability. This dynamic concentration result will in turn imply that the  algorithm produces a set of size at least $i_0$ with high probability. These error functions must be carefully chosen to satisfy certain inequalities (``variation equations"), which arise from calculations in subsequent sections. If the reader wishes to look ahead at the variation equations, they are on lines \eqref{d2vareq}, \eqref{d2rvareq}, \eqref{d3vareq}, \eqref{d3rvareq}, \eqref{qvareq}, and \eqref{c0vareq}.

\begin{align*}
 &f_3   = p^{-4}\rbrac{4t^2 + 8t+ 2 } n^{3/4} \log^3 n ,  \\ 
 &f_2   = p^{-5}\rbrac{2t^2 + 4t+ 2 } n^{1/4} \log^3 n,   \\
 &f_{3,0}   = p^{-14/3}\rbrac{5t^2 + 5t+ 1 } n^{3/4} \log^3 n,  \\
 &f_{2,0}   = p^{-17/3}\rbrac{5t^2 + 3t+ 1 } n^{1/4} \log^3 n,   \\
 &f_{1, 0} = p^{-20/3} \rbrac{t+1} n^{-1/4} \log^3 n   + p^{-1/3} \log n,\\
  &f_q =  p^{-5} \rbrac{8t+2}n^{3/4} \log^3 n. \\
\end{align*}

A note on checking the variation equations: they can be verified in a straightforward (though perhaps a bit tedious) manner as follows. On the left side of the inequality, plug in the values of all of the functions. There will be a common power of $n$, a power of $\log n$ and a power of $p$ that can be factored out of everything, leaving a factor which is just a polynomial in $t$. That polynomial will have no positive coefficients and so we just bound it using its constant term (which is valid since $t \ge 0$). For example, to verify \eqref{d3vareq}, note that

\begin{align*}
6n^{-1/2}tf_3  +6pf_2  + \frac32 n^{-1/2} t p f_q - 2n^{-1/2} f_3  ' = p^{-4}\rbrac{-24t^3-24t^2-t-4}n^{1/4} \log^3 n.
\end{align*}

Note that the bounds in \eqref{estimates} give tight estimates (i.e. accurate to within a multiplicative factor of $(1+o(1))$) for all our random variables. Indeed, first note that we have $p(t(i_0))= n^{-1/48}.$ Then, for example, note that we have $f_3 = \tilde{\Theta}\rbrac{p^{-4}n^{3/4}}= o(np^2)$ and so the first bound in \eqref{estimates} tells us that $D_{3, L} = (1+o(1)) 2np^2$. Similarly all the error terms in \eqref{estimates} are much smaller than their respective corresponding main terms.

\section{Preliminaries}

In this section we prove several bounds that are necessary for our calculations. Most of these bounds are used to justify big-O terms later. 

For $v \neq 0$, we anticipate $D_2(v) \approx D_{2,L}(v) + D_{2,R}(v)$. But to make this estimate valid we need to bound the number of $q \in Q$ such that $q \neq v$ and there are two edges $e_1, e_2 \in D_{2, L}(v) \cup D_{2, R}(v)$ such that $q \in e_1, e_2$. Note that $v$ is only in $O(1)$ many edges of size $2$, so let us assume $|e_1| = |e_2|=3$. In other words, for some $s_1, s_2 \in S$ we have $e_j = \{v, q, s_j\}$.  There are cases to consider, according to how the corresponding equations are arranged. Without loss of generality we have one of the following cases:

\begin{enumerate}

\item $v+s_1=q$ and $v+q = s_2$. Then $s_1+s_2=2v$ (so we have $\{s_1, s_2\} \in D_{1, R}(2v)$).

\item $v+s_1=q$ and $q+s_2=v$. Then $s_1+s_2=0$ (so we have $\{s_1, s_2\} \in D_{1, R}(0)$). 
\end{enumerate}
Thus, before stopping time $T$ we have 

\begin{equation}\label{d2est}
D_2(v) = D_{2,L}(v) + D_{2,R}(v) - O(1+ D_{1, R}(2v) +D_{1, R}(0))
\end{equation}

We now show that if $v \neq 0$ then $v$ does not get closed too many times. Note that $D_{1, L}(v)$ only increases in size on steps when we choose an element of $D_{2, L}(v)$, and on such steps $D_{1, L}(v)$ can increase by at most $2$. Before the stopping time $T$, we have that $\frac{D_{2, L}(v)}{Q} \le 2 n^{-1/2}t$. Thus, $C(v,i)$ is stochastically dominated by $2R$ where $R \sim Bi(i,  2n^{- 1/2} \log{1/2} n)$. An application of the Chernoff bound then tells us that $R$ does not get bigger than $\log^2 n$ w.h.p., and thus the stopping time $T$ does not happen due to the condition on $D_{1, L}(v)$. Bounding $D_{1, R}(v)$ is similar. 


We now bound the size of $D_2(v_1) \cap D_2(v_2)$, for $v_1, v_2 \neq 0, v_1 \neq \pm v_2$. Again, each of $v_1, v_2$ is only in $O(1)$ many edges of size $2$, so we assume edges have size $3$. Then for each $q \in D_2(v_1) \cap D_2(v_2)$ there is a pair $s_1, s_2 \in S$ such that both $\{v_1,s_1,q\}$ and $\{v_2,s_2,q\} $ are in $E$. There are cases to consider according to how each of the equations is arranged, but in each case we reach one of the following conclusions: $\{v_1+v_2, s_1, s_2\} \in E$, $\{v_1-v_2, s_1, s_2\} \in E$, or $\{v_2 - v_1, s_1, s_2\} \in E$. Thus, by our bounds on the sizes of sets of the form $D_{1, L}$ and $D_{1, R}$ we have, for $v_1, v_2 \neq 0, v_1 \neq \pm v_2$ that
\begin{equation}\label{d2int}
|D_2(v_1) \cap D_2(v_2)| = O(\log^2 n)
\end{equation}

To finish this section, we present a couple of lemmas which we will use several times to estimate things. The following lemma will be used to estimate fractions based on estimates of the numerator and denominator.

\begin{lemma} \label{estlemma}

For any real numbers $x, y, \epsilon_x, \epsilon_y$, if we have $x,y \neq 0$ and $\left|\frac{\epsilon_x}{x}\right|, \left|\frac{\epsilon_y}{y}\right|\le \frac{1}{2}$, then

$$\frac{x+ \epsilon_x}{y+\epsilon_y} - \frac{x}{y} = \frac{y \epsilon_x - x \epsilon_y}{y^2} + O\left(\frac{y \epsilon_x \epsilon_y + x \epsilon_y^2}{y^3} \right)$$

\end {lemma}

\begin{proof}

\begin{align*}
\frac{x+ \epsilon_x}{y+\epsilon_y} - \frac{x}{y} &= \frac{x}{y} \left\{ \left(1+\frac{\epsilon_x}{x} \right) \cdot \frac{1}{1+\frac{\epsilon_y}{y}} -1 \right\}\\
&= \frac{x}{y} \left\{ \left(1+\frac{\epsilon_x}{x} \right) \cdot \left[1 - \frac{\epsilon_y}{y} + O\left(\frac{\epsilon_y^2}{y^2} \right) \right] -1 \right\}\\
&= \frac{x}{y} \left\{ \frac{\epsilon_x}{x} - \frac{\epsilon_y}{y} + O\left(\frac{\epsilon_x \epsilon_y}{x y}+\frac{\epsilon_y^2}{y^2} \right)\right\}\\
&= \frac{y \epsilon_x - x \epsilon_y}{y^2} + O\left(\frac{y \epsilon_x \epsilon_y + x \epsilon_y^2}{y^3} \right)
\end{align*}

\end{proof}

\section{Tracking the \texorpdfstring{$D_2$}{} variables}

In this section we prove dynamic concentration of the $D_{2}$-type variables. We start with  $D_{2,L}(v)$ for $v \neq 0$.

  For an arbitrary random variable $V$ we define
\[ \Delta V(i) = V(i+1) - V(i). \]
We let $ {\mathcal F}_i $ be the filtration of the probability space
given by the first $i$ steps of the process.
We need to estimate $E[ \Delta D_{2,L}(v) | \mathcal{F}_i]$ for  $ i \le T$.  The one step change $\Delta D_{2,L}(v)$ has both positive and negative contributions. $D_{2,L}(v)$ gains a pair $\{v,q\}$ when we choose an element $q'$ such that $\{v,q,q'\} \in D_{3,L}(v)$, except in the case where we happen to have $q \in D_{2}(q')$ (in which case we may conclude that either $2q'$ or $2q$ is in $D_2(v)$). 
  $D_{2,L}(v)$ loses a pair $\{v,b\}$ when $b$ is chosen or closed.

Thus we can put

$$E[ \Delta D_{2,L}(v) | \mathcal{F}_i] = \frac{1}{Q} \left\{ 2D_{3,L}(v) +O(n^{1/2} tp) - \sum_{\{v,q, s\}\in D_{2, L}(v)} D_{2}(q) \right\}$$

In the sum above, it is intended that $q \in Q, s \in S$, so we are summing all the ways to lose an edge in $\{v,q, s\} \in D_{2, L}(v)$ due to closing $q$. The big-O term absorbs the error due to: the possibility that we choose some $q \in \{v, q, q'\} \in D_{3, L}(v)$ such that $q' \in D_2(q)$; the possibility of losing $\{v,q, s\}$ due to choosing $q$;  and the effect of the $O(1)$ many edges of size $2$  in $D_{2, L}(v)$. 

Almost all of the terms in the sum are $D_2(q) \approx D_{2, L}(q) + D_{2, R}(q) \approx 3n^{1/2} tp$, by \eqref{d2est} and our estimates on the variables we're tracking. However, we may have some terms with $q$ such that $-q \in S$, in which case $D_{2,L}(q)=0$. However, supposing that $\{v,q, s\}\in D_{2,L}(v)$  and $-q \in S$, we conclude (by considering cases) that either $\{v, s, -q\} \in D_{1,R}(v)$ or $\{-v, s, -q\} \in D_{1, R}(-v)$. Thus there are $O(\log^2 n)$ many such terms in the sum. Now applying \eqref{d2est}, our estimate for $Q$, lemma  \ref{estlemma}, and recalling  $t = O(\log^{1/2} n)$ we get

\begin{eqnarray}
&& E[ \Delta D_{2,L}(v) | \mathcal{F}_i] =   2p - 3t^2 p + 2n^{-1}p^{-1} f_3  + \rbrac{1+\frac32 t^2} n^{-1} f_q + 6n^{-1/2}tf_2 \nonumber   \\
&& \qquad + \tilde{O}\left(n^{-2}p^{-2}f_q f_3 +n^{-3/2}p^{-1}f_q f_2+ n^{-1}p^{-1}f_2 ^2 + n^{-1/2}  + n^{-1/2}  D_{1, R}(0) \right) \nonumber\\
&&=2p - 3t^2 p+ 2n^{-1}p^{-1} f_3     + \rbrac{1+\frac32 t^2} n^{-1}  f_q  + 6n^{-1/2}tf_2+ \tilde{O}\left(n^{-1/2}  p^{-11} \right).\label{d2trend}
\end{eqnarray}



For each vertex $v \neq 0$  , we define the sequence of random variables

\begin{displaymath}
   D_{2,L}^+(v,i)  := \left\{
     \begin{array}{lr}
       D_{2,L}(v,i) - 2n^{1/2} tp - 2f_2  & : \pm v \notin S(i)\\
       D_{2,L}^+(v,i-1) & : otherwise
     \end{array}
   \right.
\end{displaymath} 
We will show that the sequence of random variables $D_{2,L}^+(v,0) \ldots D_{2,L}^+(v,T)$ is a supermartingale, and then use a deviation inequality to show that w.h.p. $D_{2,L}^+(v,i)$ is never positive (and hence $D_{2,L}(v,i)$ does not violate its upper bound). For $ i < T$

\begin{eqnarray*}
&& E[ \Delta D^+_{2,L}(v) | \mathcal{F}_i]  \le  2n^{-1}p^{-1} f_3  + \rbrac{1+\frac32 t^2} n^{-1} f_q  + 6n^{-1/2}tf_2  -2n^{-1/2}f_2 '\\
&& \qquad + \tilde{O}\left(n^{-1/2}  p^{-11}  + n^{-1}f_2 ''\right)\\
&& \le -\tilde{\Omega} \left(n^{-1/4}p^{-5}   \right)
\end{eqnarray*}
Note that we have used \eqref{d2trend}, and approximated the $1$-step change of deterministic functions using derivatives (i.e. Taylor's theorem). The last line can be verified by observing that 
\begin{equation}
2n^{-1}p^{-1} f_3  + \rbrac{1+\frac32 t^2} n^{-1} f_q + 6n^{-1/2}tf_2 -2n^{-1/2}f_2 '\le -2p^{-5}n^{-1/4}\log^3 n.\\ \label{d2vareq} 
\end{equation}

Now we use Azuma-Hoeffding to bound the probability that the supermartingale \newline $D_{2,L}^+(v,0) \ldots D_{2,L}^+(v,T)$ becomes positive.

\begin{lemma}Let $X_j$ be a supermartingale, with $|\Delta X_i| \leq c_i$ for all $i$. Then

$$P(X_m - X_0 \geq a) \leq \displaystyle \exp\left(-\frac{a^2}{2 \displaystyle \sum_{i\leq m} c_i^2 }\right).$$ \end{lemma}
We now bound the largest possible $1$-step changes in the supermartingale. By examining all the contributions (positive and negative) we see that the largest possible $1$-step change is a large negative contribution due to the algorithm choosing some vertex $v'$ such that $D_2(v')$ contains a lot of vertices $q$ in an edge $\{v, q, s\} \in D_{2, L}(v)$. Using \eqref{d2int} to bound the number of such $q$ (for any fixed $v'$) we get

$$\left| \Delta D_{2,L}^+(v,i)  \right| =O( \log^2 n) $$
Thus, applying Azuma-Hoeffding we see that the event that the supermartingale $D_{2,L}^+(v)$ ever becomes positive has probability at most 

$$\exp\left\{ - \Omega \left( \frac{(n^{1/4} \log^{3} n )^2}{ i_0 (\log^2 n)^2} \right) \right\} = o(n^{-1}).$$
As there are at most $O(n)$ such supermartingales, with high probability none of them have such a large upward deviation, and $D_{2,L}(v)$ stays below its upper bound for all $v$. 

The lower bound for $D_{2,L}(v)$ is similar, as are the bounds for $D_{2,R}(v)$ for $v \neq 0$.

Now we address $D_{2,R}(0)$. We have

$$  E[ \Delta D_{2,R}(0) | \mathcal{F}_i] = \displaystyle \frac{1}{Q}\left\{ 2 D_{3,R}(0)+ O(n^{1/2} tp^{1/3}) -\sum_{\{0,q,s\}\in D_{2,R}(0)} D_{2}(q)  \right\},  $$
and note that for every $\{0,q,s\}\in D_{2,R}(0)$ we have that $-q =s \in S$ and so $D_2(q) = D_{2, R}(q)$. Thus

\begin{eqnarray}
&& \qquad E[ \Delta D_{2,R}(0) | \mathcal{F}_i] \nonumber \\
&&=p^{1/3} - \frac{1}{2}t^2p^{1/3} + \frac{1}{2}n^{-1/2}t f_{2, 0} + n^{-1}p^{-1}f_{3, 0} + \rbrac{\frac12 + \frac34 t^2}p^{-2/3}f_q + \frac{1}{2}n^{-1/2}tp^{-2/3}f_2 \nonumber\\
&& \qquad   + \tilde{O}\left(n^{-2}p^{-2}f_{3,0}f_q + n^{-3/2}p^{-5/3}f_2f_q + n^{-1}p^{-1}f_{2, 0}f_2 + n^{-1/2}p^{-1}f_{2, 0}f_q \right) \nonumber\\
&&  = n^{-1}p^{-1}f_{3, 0} + \rbrac{\frac12 + \frac34 t^2}p^{-2/3}f_q + \frac{1}{2}n^{-1/2}tp^{-2/3}f_2 + \tilde{O}\left(n^{-1/2}p^{-35/3} \right)    .    \label{d2rtrend}
\end{eqnarray}

We define the sequence of random variables 

$$D_{2,R}^+(0,i) := D_{2,R}(0,i) - n^{1/2} tp^{1/3} - f_{2, 0}.$$
We will show that the sequence $D_{2,R}^+(0,0) \ldots D_{2,R}^+(0,T)$ is a supermartingale. For $ i < T$,

\begin{eqnarray*}
&& E\left[ \Delta D^+_{2,R}(0,i)| \mathcal{F}_i \right] \leq  \frac{1}{2}n^{-1/2}t f_{2, 0}+ n^{-1}p^{-1}f_{3, 0} +\rbrac{\frac12 + \frac34 t^2}p^{-2/3}f_q + \frac{1}{2}n^{-1/2}tp^{-2/3}f_2 \\
&& \qquad  -n^{-1/2}f_{2, 0}'+ \tilde{O}\left(n^{-1/2}p^{-35/3}  + n^{-1}f_{2, 0}'' \right)\\
&& \le -\tilde{\Omega} \left(n^{-1/4}  \cdot p^{-17/3} \right)
\end{eqnarray*}
Note that in the first line we have used \eqref{d2rtrend}. The last line can be verified by observing that 
\begin{align}
& \frac{1}{2}n^{-1/2}t f_{2, 0}+ n^{-1}p^{-1}f_{3, 0}+  \rbrac{\frac12 + \frac34 t^2}p^{-2/3}f_q  + \frac{1}{2}n^{-1/2}tp^{-2/3}f_2  -n^{-1/2}f_{2, 0}'\nonumber\\
&\qquad \le -p^{-17/3} n^{-1/4} \log^3 n\label{d2rvareq}
\end{align} 

Again, the biggest possible $1$-step changes in the above supermartingale comes from the intersections of $D_2$ sets. Using \eqref{d2int},

$$\left| \Delta D^+_{2,R}(0,i)  \right| =O( \log^2 n )$$
Thus, the event that the supermartingale $D^+_{2,R}(0,i)$ ever becomes positive has probability at most

$$\exp\left\{ - \Omega \left( \frac{(n^{1/4} \log^{3} n )^2}{ i_0 (\log^2 n)^2} \right) \right\}  = o(1).$$
Therefore with high probability $D_{2,R}(0)$ stays below its upper bound. Proving the lower bound for $D_{2,R}(0)$ is similar. 

\section{Tracking the \texorpdfstring{$D_3$}{} variables}

The variable $D_{3,L}(v)$ is nonincreasing, and we lose a triple $\{v,q,q'\}\in D_{3,L}(v)$ whenever $q$ or $q'$ is closed or chosen.
Thus we have 
$$ E[ \Delta D_{3,L}(v) | \mathcal{F}_i] =- \displaystyle \frac{1}{Q} \sum_{\{v,q,q'\}\in D_{3,L}(v)}  D_2(q) \cup D_2(q') \cup \{q,q'\}  $$
Before the stopping time $T$, we can estimate the above expression. Most of the terms in the sum will have $-q, -q' \notin S$.   Supposing that $\{v,q,q'\}\in D_{3,L}(v)$ and $-q \in S$, we conclude that one of $q'$ or $-q'$ must be in $D_2(v)$. Thus there are $O \left(n^{1/2}tp \right)$ many such terms in the sum. Now applying \eqref{d2est}, our control on the variables before the stopping time $T$, lemma  \ref{estlemma}, and recalling  $t = O(\log^{1/2} n)$ we get

  \begin{eqnarray}
&&E[ \Delta D_{3,L}(v) | \mathcal{F}_i] =- 6n^{1/2}t p^2+ 6n^{-1/2}tf_3+ 6pf_2 + \frac32 n^{-1/2} t p f_q \nonumber \\
&&  \;\;\;\;\;\;\;\;\;\;+ \tilde{O}\left(p D_{1, R}(0) + p  + n^{-1}f_2f_q + n^{-3/2} p^{-1} f_3f_q + n^{-1}p^{-1}f_2 f_3  \right)\nonumber\\ 
&&= - 6n^{1/2}t p^2 + 6n^{-1/2}tf_3 + 6pf_2 + \frac32 n^{-1/2} t p f_q  +\tilde{O}\left( p^{-10} \right)    \label{d3trend}
\end{eqnarray}

We define the sequence of random variables 

$$D_{3,L}^+(v,i) := D_{3,L}(v,i) - 2n p^2 - f_{3}$$
We will show that the sequence $D_{3,L}^+(v,0) \ldots D_{3,L}^+(v,T)$ is a supermartingale. For $i < T$, we have the inequality

  \begin{eqnarray*}
&& E\left[ \Delta D^+_{3,L}(v,i)| \mathcal{F}_i \right] \le   6n^{-1/2}tf_3   +6pf_2 + \frac32 n^{-1/2} t p f_q - 2n^{-1/2} f_3  '\\
&&  \;\;\;\;\;\;\;\;\;\;+ \tilde{O}\left(p^{-10}   + n^{-1}f_3  ''\right)\\
&& \le -\tilde{\Omega} \left(p^{-4} n^{1/4}  \right)
\end{eqnarray*}
Note that in the first line we have used \eqref{d3trend}. The last line can be verified by observing that  
\begin{equation}
6n^{-1/2}tf_3  +6pf_2  + \frac32 n^{-1/2} t p f_q - 2n^{-1/2} f_3  ' \le -4p^{-4}n^{1/4} \log^3 n  \label{d3vareq} 
\end{equation}

Now we use Azuma-Hoeffding to bound the probability that the supermartingale becomes positive. By considering the total number of elements that get closed in any one step, we see

$$\left| \Delta D^+_{3,L}(v,i)  \right| \leq  O\left(n^{1/2}tp \right) \le O\left(n^{1/2} \log^{1/2}n \right).$$
Thus,applying Azuma-Hoeffding we see that the event that the supermartingale $D^+_{3,L}(v,i)$ is ever positive has probability at most

$$\exp\left\{ -\Omega\left(\frac{(n^{3/4} \log^{3} n )^2}{  \displaystyle  i_0 (n^{1/2} \log^{1/2}n)^2} \right)\right\}  = o(n^{-1})$$
As there are $O( n)$ such supermartingales, with high probability $D_{3,L}(v)$ stays below its upper bound for all $v$. 

The lower bound for $D_{3,L}(v)$ is similar, as are the bounds for $D_{3,R}(v)$ for $v \neq 0$.

Now we address $D_{3,R}(0)$. Using the same methods to estimate $E[ \Delta D_{3,L}(v) | \mathcal{F}_i]$ before the stopping time $T$, we get

  \begin{eqnarray*}
&& E[ \Delta D_{3,R}(0) | \mathcal{F}_i] = \displaystyle - \frac{1}{Q}\sum_{\{0,q, -q\}\in D_{3,R}(0)} D_{2}(q) \cup D_{2}(-q) \cup \{q, -q\} 
 \end{eqnarray*}
but note that each term in the sum has $D_{2, L}(q) = D_{2, L}(-q)$ and so the term is roughly $4n^{1/2}tp$. Thus

  \begin{eqnarray}
&&E[ \Delta D_{3,R}(0) | \mathcal{F}_i]  =- 2n^{1/2}tp^{4/3}  +2n^{-1/2}tf_{3, 0} + 2p^{1/3}f_2 + n^{-1/2}tp^{1/3}f_q \nonumber\\
&& \qquad + O\left(p^{1/3}D_{1, R}(0) + n^{-1}p^{-2/3}f_2f_q + n^{-3/2}p^{-1}f_{3,0}f_q+ n^{-1}p^{-1} f_2 f_{3, 0} \right) \nonumber\\
&&= - 2n^{1/2}tp^{4/3}  +2n^{-1/2}tf_{3, 0}+2p^{1/3}f_2 + n^{-1/2}tp^{1/3}f_q + \tilde{O}\left(p^{-32/3}  \right)  \label{d3rtrend}
\end{eqnarray}
We define the sequence of random variables 

$$D_{3,R}^+(0,i) := D_{3,R}(0,i) - n p^{4/3} - f_{3, 0}.$$
for $ i < T$

  \begin{eqnarray*}
&& E\left[ \Delta D^+_{3,R}(0,i)| \mathcal{F}_i \right] \le   2n^{-1/2}tf_{3, 0} +2p^{1/3} f_2 + n^{-1/2}tp^{1/3}f_q - n^{-1/2} f_{3, 0}'\\
&& \qquad + \tilde{O}\left(p^{-32/3}     + n^{-1}f_{3, 0}''\right)\\
&& \le -\tilde{\Omega} \left(p^{-14/3} n^{1/4}  \right).
\end{eqnarray*}
Note that in the first line we have used \eqref{d3rtrend} and the fact that $D_{3,R}(0)$ is in the critical interval. The last line can be verified by observing that  
\begin{equation}
 2n^{-1/2}tf_{3, 0} +2p^{1/3} f_2 + n^{-1/2}tp^{1/3}f_q - n^{-1/2} f_{3, 0}' \le -p^{14/3}n^{1/4} \log^3 n . \label{d3rvareq} 
\end{equation}

Now we use Azuma-Hoeffding to bound the probability that the supermartingale becomes positive. We have

$$\left| \Delta D^+_{3,R}(0,j)  \right| \leq  O\left(n^{1/2}tp \right)$$
Thus, applying Azuma-Hoeffding we see that the event that the supermartingale $ D^+_{3,R}(0)$ is ever positive has probability $o(1)$. Thus with high probability none of them have such large upward deviations, and $ D_{3,R}(0)$ stays below its upper bound. The lower bound for $ D_{3,R}(0)$ is similar.

\section{Tracking the \texorpdfstring{$Q$}{} variable}

We have
$$E[\Delta Q | \mathcal{F}_i] =  -1-\frac{1}{Q} \sum_{q \in Q} D_2(q)$$
and so 
\begin{equation}
    |E[\Delta Q | \mathcal{F}_i] - 3n^{1/2}tp| \le 3f_2 + O\rbrac{ D_{1,R}(0) + \log^2 n} \label{qtrend}
\end{equation}
 
We define the sequence of random variables $Q^+ (0) \ldots Q^+ (T)$, where 

$$Q^+ (i):=Q(i) - 2np - f_q $$
We will choose the function $f_q$ so that $Q^+$ is a supermartingale. We have the inequality

  \begin{align*}
 E\left[\Delta Q^+ | \mathcal{F}_i \right] &\leq  3f_2 - n^{-1/2} f_q ' + \tilde{O}(1 + p^{-2/3} + n^{-1}f_q '')\\
& \le -\tilde{\Omega} \left(   p^{-5} n^{1/4} \right)
\end{align*}
Note that in the first line we have used \eqref{qtrend}. The last line can be verified by observing that  
\begin{equation}
3f_2 - n^{-1/2} f_q ' \le -2 p^{-5}n^{1/4} \log^3 n .\label{qvareq} 
\end{equation}

Now we use Azuma-Hoeffding to bound the probability that the supermartingale becomes positive. We have

$$\left| \Delta Q^+(i_0,j)  \right| \leq  O(f_2 ) = \tilde{O}(p^{-5}n^{1/4})= \tilde{O}(n^{17/48}).$$
Thus, applying Azuma-Hoeffding we see that the event that the supermartingale $Q^+$ has such a large upward deviation has probability at most

$$\exp\left\{ - \tilde{\Omega}\left(\frac{(n^{3/4}  )^2}{  i_0 \cdot (n^{17/48})^2} \right)\right\}  = o(1)$$
Thus w.h.p. $Q$ stays below its upper bound. The lower bound for $Q$ is similar. \newline

\section{Tracking \texorpdfstring{$D_{1, R}(0)$}{} }

We have the expected $1$-step change

$$  E[ \Delta D_{1, R}(0) | \mathcal{F}_i] = \displaystyle \frac{D_{2,R}(0)}{Q}   $$
and so before $T$,

 \begin{eqnarray*}
&& E[ D_{1, R}(0) | \mathcal{F}_i] = n^{-1/2}tp^{-2/3}+ \frac{1}{2}n^{-1}p^{-1}f_{2, 0} + \frac14 n^{-3/2}tp^{-5/3} f_q+ O\left(n^{-2}p^{-2} f_q f_{2,0}  \right)
\end{eqnarray*}

We define the sequence of random variables 

$$D^+_{1, R}(0) := D_{1, R}(0) - \frac{1}{2}  p^{-2/3} - h$$
where 

$$h:=n^{-1/4} \log^3 n  (t+1)p^{-20/3}.$$
we have

  \begin{eqnarray*}
&& E\left[ \Delta D^+_{1, R}(0)| \mathcal{F}_i \right]\\
&& \leq  \frac{1}{2}n^{-1}p^{-1}f_{2, 0} + \frac14 n^{-3/2}tp^{-5/3} f_q -n^{-1/2}h'\\
&&\;\;\;\;\;\;\;\;\;\;+ O\left(n^{-2}p^{-2} f_q f_{2,0} + n^{-1}h'' \right)\\
&& \le -\tilde{\Omega} \left(n^{-3/4}  p^{-5/3}   \right)
\end{eqnarray*}
The last line can be verified by observing that 
\begin{equation}
 \frac{1}{2}n^{-1}p^{-1}f_{2, 0} + \frac14 n^{-3/2}tp^{-5/3} f_q -n^{-1/2}h' \le -\frac{1}{2}n^{-3/4} \log^3 n p^{-20/3 } .\label{c0vareq} 
\end{equation}  We will apply the following inequality due to Freedman \cite{F75}. 

\begin{lemma}\label{lem:Freedman}
Let $X_i$ be a supermartingale, with $\Delta X_i \leq C$ for all $i$, and $V(i) :=\displaystyle \sum_{k \le i} Var[\Delta X(k)| \mathcal{F}_{k}]$  Then
$$P\left[\exists i: V(i) \le v, X_i - X_0 \geq d \right] \leq \displaystyle \exp\left(-\frac{d^2}{2(v+Cd) }\right).$$ \end{lemma}
For our application of this inequality, we can use $C=1$. To determine a suitable value for $v$, note first that before $T$ we have

  \begin{eqnarray*}
&&Var[\Delta D^+_{1, R}(0, k)| \mathcal{F}_{k}] = Var[\Delta D^+_{1, R}(0, k)| \mathcal{F}_{k}] \le E \left[ \left(\Delta D_{1, R}(0, k) \right)^2 | \mathcal{F}_{k} \right]\\
&& = \frac{D_{2,R}(0,k)}{Q(k)} < n^{-1/2}t(k)p(t(k))^{-2/3}
\end{eqnarray*}
so we bound the sum

 \begin{eqnarray*}
&&\displaystyle \sum_{k \le i} n^{-1/2}t(k)p(t(k))^{-2/3} \le  \int_0^{t(i)}  \tau p(\tau)^{-2/3} d \tau <  p(t(i))^{-2/3} 
\end{eqnarray*}
(note that $t p^{-2/3}$ is increasing in $t$ so we may bound the sum with an integral) so we set $v= p(t(i))^{-2/3}$. Now we see by Lemma~\ref{lem:Freedman} that with $d= p(t(i))^{-1/3} \log n$, with high probability the supermartingale $D_{1, R}^+(0,i)$ is no larger than $d$. Therefore before $T$ we have the upper bound

$$D_{1, R}(0, i) \le \frac{1}{2}  p^{-2/3} + h + p^{-1/3} \log n.$$
The lower bound for $D_{1, R}(0,i)$ is similar. 

In particular,  $D_{1, R}(0,i_0) = \frac{1}{2} p(t(i_0))^{-2/3} (1+o(1)) = \Theta\left(n^{1/72}\right)$, while $D_{1, R}(v, i_0), D_{1, L}(v, i_0) = O\left(\log^2 n \right)$ for all $v \neq 0$. The behavior of $D_{1, R}(0)$, and the fact that $D_{1, R}(0)$ appears in many of our big-O terms, would seem to indicate that some new ideas would be needed to track the sum-free process much further (e.g. to prove Conjecture \ref{conj}).\newline
{\bf Acknowledgements:} The author would like to thank Tom Bohman for several helpful discussions

\bibliographystyle{plain}
\bibliography{refs}

\begin{thebibliography}{10}

\bibitem{BaBe}
Deepak {Bal} and Patrick {Bennett}.
\newblock {The bipartite $K_{2,2}$-free process and bipartite Ramsey number
  $b(2, t)$}.
\newblock {\em arxiv.org/abs/1808.02139}, 2019.

\bibitem{BMSsumfree}
J\'ozsef Balogh, Robert Morris, and Wojciech Samotij.
\newblock Random sum-free subsets of abelian groups.
\newblock {\em Israel J. Math.}, 199(2):651--685, 2014.

\bibitem{Hind}
Patrick Bennett and Tom Bohman.
\newblock A note on the random greedy independent set algorithm.
\newblock {\em Random Structures Algorithms}, 49(3):479--502, 2016.

\bibitem{BDZ}
Patrick {Bennett}, Andrzej {Dudek}, and Shira {Zerbib}.
\newblock {Large triangle packings and Tuza's conjecture in sparse random
  graphs}.
\newblock {\em arxiv.org/abs/1810.11739}, 2018.

\bibitem{Bo10}
Tom Bohman.
\newblock The triangle-free process.
\newblock {\em Adv. Math.}, 221(5):1653--1677, 2009.

\bibitem{Hfree}
Tom Bohman and Peter Keevash.
\newblock The early evolution of the {$H$}-free process.
\newblock {\em Invent. Math.}, 181(2):291--336, 2010.

\bibitem{BKtridynamic}
Tom Bohman and Peter Keevash.
\newblock Dynamic concentration of the triangle-free process.
\newblock In {\em The {S}eventh {E}uropean {C}onference on {C}ombinatorics,
  {G}raph {T}heory and {A}pplications}, volume~16 of {\em CRM Series}, pages
  489--495. Ed. Norm., Pisa, 2013.

\bibitem{Bowa}
Tom {Bohman} and Lutz {Warnke}.
\newblock {Large girth approximate Steiner triple systems}.
\newblock {\em arxiv.org/abs/1808.01065}, 2019.

\bibitem{FGM}
Gonzalo {Fiz Pontiveros}, Simon {Griffiths}, and Rob {Morris}.
\newblock {The triangle-free process and the Ramsey number $R(3,k)$}.
\newblock {\em arxiv.org/abs/1302.6279}, 2013.

\bibitem{F75}
David~A. Freedman.
\newblock On tail probabilities for martingales.
\newblock {\em Ann. Probability}, 3:100--118, 1975.

\bibitem{P1}
Michael~E. Picollelli.
\newblock The diamond-free process.
\newblock {\em Random Structures Algorithms}, 45(3):513--551, 2014.

\bibitem{P3}
Michael~E. Picollelli.
\newblock The final size of the {$C_\ell$}-free process.
\newblock {\em SIAM J. Discrete Math.}, 28(3):1276--1305, 2014.

\bibitem{Wa2}
Lutz Warnke.
\newblock The {$C_\ell$}-free process.
\newblock {\em Random Structures Algorithms}, 44(4):490--526, 2014.

\bibitem{Wa3}
Lutz Warnke.
\newblock When does the {$K_4$}-free process stop?
\newblock {\em Random Structures Algorithms}, 44(3):355--397, 2014.

\bibitem{Wanote}
Lutz {Warnke}.
\newblock {On Wormald's differential equation method}.
\newblock {\em arxiv.org/abs/1905.08928}, 2019.

\bibitem{Wz2}
Guy {Wolfovitz}.
\newblock {The $K_4$-free process}.
\newblock {\em arxiv.org/abs/1008.4044}, 2010.

\bibitem{GW}
Guy Wolfovitz.
\newblock Triangle-free subgraphs in the triangle-free process.
\newblock {\em Random Structures Algorithms}, 39(4):539--543, 2011.

\bibitem{nick2}
Nicholas~C. Wormald.
\newblock The differential equation method for random graph processes and
  greedy algorithms.
\newblock In {\em Notes on lectures at the Summer School on Randomized
  Algorithms at Antonin, Poland}, 1997.

\end{thebibliography}

\end{document}